\documentclass[final,5p,times,twocolumn]{elsarticle}

\usepackage[utf8]{inputenc}
\usepackage[T1]{fontenc}
\usepackage{lmodern}
\usepackage{etex}
\reserveinserts{28}
\usepackage{amsthm}
\usepackage{amssymb}
\usepackage{amsmath}
\usepackage{float}
\usepackage{tikz}
\usetikzlibrary{arrows,calc,shapes,decorations.pathreplacing}
\usepackage{hyperref}
\hypersetup{pdftex,colorlinks=true,allcolors=blue}

\newtheorem{theorem}{Theorem}
\newtheorem{lemma}[theorem]{Lemma}
\newtheorem{proposition}[theorem]{Proposition}

\newtheorem*{theorem*}{Theorem}

\def\P{\mathrm{P}}
\def\E{\mathrm{E}}

\begin{document}

\begin{frontmatter}



\title{Lowest priority waiting time distribution in an accumulating priority L\'{e}vy queue}


\author[label1]{Offer Kella}
\ead{offer.kella@huji.ac.il}
\author[label2]{Liron Ravner\corref{cor1}}
\ead{lravner@post.tau.ac.il}
\address[label1]{Department of Statistics, The Hebrew University of Jerusalem, Israel}
\address[label2]{Department of Statistics and Operations Research, Tel Aviv University, Israel}
\cortext[cor1]{Corresponding author.}

\begin{abstract}
We derive the waiting time distribution of the lowest class in an accumulating priority (AP) queue with positive L\'{e}vy input. The priority of an infinitesimal customer (particle) is a function of their class and waiting time in the system, and the particles with the highest AP are the next to be processed. To this end we introduce a new method that relies on the construction of a workload overtaking process and solving a first-passage problem using an appropriate stopping time.
\end{abstract}


\end{frontmatter}

\section{Introduction}\label{sec:intro}

Suppose that a single server provides service to non-atomic (infinitesimal) customers, that we refer to as particles, of different types according to a dynamic non-preemptive accumulating priority (AP) service regime. That is, the priority of every particle in the queue is a function of their class and their accumulated waiting time in the queue. Upon a service completion the server admits the particle with the highest accumulated priority. Such regimes are common in health-care applications where the condition of a patient can deteriorate while waiting (e.g. \cite{SSTZ2014}). The analysis of the accumulating priority M/G/1 queue goes back to \cite{K1964}. Note that it was then called the delay-dependent priority regime, but this has come to mean different things over the years and so we opt to use the accumulating-priority terminology of \cite{STZ2014}. This paper analyses the waiting time distribution of the particle class with the lowest AP rate in the general setting of a L\'{e}vy driven queue with positive input (see \cite{DM2012}) by means of a novel method. In this setting particles may arrive as a continuous flow or in batches with a whole mass of other particles. Thus, this formulation can also be useful for inventory or insurance models that include heterogeneous input and priorities. The method introduced here relies on a martingale representation of the composite workload accumulation process brought on by customer overtaking, and the subsequent solution of a first-passage time problem.

The M/G/1 queue with a linear AP regime has been extensively studied. The expected waiting times are known to satisfy a recursive formula, (3.47) on p. 131 of \cite{book_K1976}, which we will refer to as the Kleinrock formula. The initial condition for the recursion is the expected waiting time of the lowest priority classes, which can be computed on its own. The mean-value analysis was extended to other AP functions (without closed form solutions such as the Kleinrock formula): e.g. power law \cite{KF1967}, affine \cite{G1977}, and concave \cite{NA1979}. In \cite{STZ2014} the distributions of the waiting times for all customer classes were derived as a system of recursive equations for the LST of the different classes. This was done by the construction of an auxiliary process, the maximum-priority process, and the derivation of its LST. The distributional analysis was extended to a multi-server model in \cite{SSTZ2014} and \cite{LS2016} (the latter included heterogeneous servers), with the additional assumption that service times are exponential. The method of \cite{STZ2014} was also shown to be useful for certain non-linear accumulating functions in \cite{L2015}, and for a preemptive piriority regime in \cite{FD2015}. In \cite{SB2013} the lowest priority waiting time distribution was derived for a dynamic priority model in which customers jump to a higher priority class after waiting for a certain time threshold. Economic analysis of this model where customers can purchase their AP rates appeared in \cite{HR2016}.

Our analysis is the first to adress the AP queue in the general case of a L\'{e}vy driven queue, i.e.\ the workload arrival process is a positive L\'{e}vy process and not necessarily a compound Poisson process as in the regular M/G/1. This means that the incoming workload can be any subordinator, that is, a non-decreasing L\'{e}vy process. To acheive this we suggest an alternative approach to derive the distribution of the waiting time of the lowest priority customer class. An up to date review of the research of L\'{e}vy driven queues can be found in \cite{DM2012} and in greater detail in \cite{book_DM2015}. Our method involves constructing an accumulative-overtaking process and a stopping time with respect to it that is distributed as the total waiting time. A Wald-type martingale is then used to derive the LST of the waiting time.

In the sequel we establish a decomposition result for the stationary distribution of the workload that will be served before an arriving particle. Such a particle may arrive as part of a batch of particles and thus it will face more workload then just the amount that was present a moment before the batch arrived. This distinction will be important when analysing the overtaking process brought on by the accumulating priority regime in the following sections.

\section{Preliminary technicalities}
This somewhat more technical section is needed to justify some of the analysis that appears later. Let $J_0=\{J_0(t)|t\ge0\}$ be a subordinator, i.e.\ non-decreasing right continuous L\'evy process (see p. 71 of \cite{book_B1996}), with respect to some filtration $\{\mathcal{F}_t|t\ge 0\}$ satisfying the usual conditions (right continuous, augmented). This means that $J_0(t)\in\mathcal{F}_t$ for all $t\ge 0$ and that $J_0(t+s)-J_0(t)$ is independent of $\mathcal{F}_t$ for all $s,t\ge 0$. In addition, we assume that $\rho_0\equiv EJ_0(1)<\infty$.

As in \cite{KY2013}, we recall that there are a constant $c_0\ge 0$, a L\'evy measure $\nu_0$ satisfying $\int_{(0,\infty)}x\nu_0(dx)<\infty$, and a Poisson random measure $N_0$ on $(0,\infty)\times[0,\infty)$ with mean measure $\nu_0\otimes\ell$, where $\ell$ is Lebesgue measure, such that
\begin{equation*}
J_0(t)=c_0 t+\int_{[0,t]\times(0,\infty)}xN_0(dx,ds)\ .
\end{equation*}
Also, as in (23), (25) and Lemma~1 of \cite{KY2013} it follows that if $W$ is some non-negative c\`adl\`ag adapted process and $F$ is a continuous function for which $\int_{(0,\infty)}(F(w+x)-F(w))^2\nu_0(dx)$ is bounded (in $w$) on $[0,\infty)$, then
\begin{align*}
M(t)=\int_0^t\int_{(0,\infty)}F(W(s-)+x)-F(W(s-))N_0(dx,ds)\nonumber\\
-\int_0^t\int_{(0,\infty)}F(W(s)+x)-F(W(s)))\nu_0(dx)ds
\end{align*}
is a zero mean $L^2$ martingale satisfying $M(t)/t\to0$ both a.s. and in $L^2$. In particular, if $F$ is differentiable with bounded derivative $f$ and we denote by $Y_e$ an independent (of all other processes) random variable with
\begin{equation*}
P(Y_e\le t)=\frac{c_0+\int_0^t\nu_0(y,\infty)dy}{\rho_0}\ ,
\end{equation*}
(see (4.6) of \cite{K1998})
then it follows with $g(w)=Ef(w+Y_e)$ that
\begin{align*}
c_0 f(W(s))&+\int_{(0,\infty)}(F(W(s)+x)-F(W(s))\nu_0(dx)\nonumber\\
&=c_0 f(W(s))+\int_{(0,\infty)}\int_0^xf(W(s)+y)dy\nu_0(dx)\\
&=c_0 f(W(s))+\int_0^\infty f(W(s)+y)\nu_0(y,\infty)dy\\
&=\rho_0 g(W(s))\ .\nonumber
\end{align*}
With all of the above we have, since $J_0(t)/t\to \rho_0$ a.s., that
\begin{equation}\label{eq:cumulative}
\begin{split}
&\frac{1}{J_0(t)}\int_0^t\frac{1}{\Delta J_0(s)}\int_0^{\Delta J_0(s)}f(W(s-)+x)dxdJ_0(s)\\
&-\frac{1}{t}\int_0^tg(W(s))ds
\end{split}
\end{equation}
converges almost surely and in $L^2$ to zero, where for the case $\Delta J_0(s)=0$ we define by convention
\begin{equation*}
\frac{1}{\Delta J_0(s)}\int_0^{\Delta J_0(s)}f(W(s-)+x)dx\equiv f(W(s-))\ .
\end{equation*} 
Now, note that
\begin{align*}
&\int_0^t\frac{1}{\Delta J_0(s)}\int_0^{\Delta J_0(s)}f(W(s-)+x)dxdJ_0(s)\nonumber\\
&=
c_0 \int_0^tf(W(s))ds+\sum_{0<s\le t}\int_0^{\Delta J_0(s)}f(W(s-)+x)dx\ ,
\end{align*}
and observe that if $W(\cdot)$ is some content that is found at time $t$, then the right hand side aggregates the function values of the content in front of all the $J_0(\cdot)$ particles that arrived by time $t$. That is, if a particle arrives on its own, then the content in front of it is just $W(s-)=W(s)$. If it is in some $x\ge 0$ location in a batch (=jump) then the amount is $W(s-)+x$. If we divide the right hand side by $J_0(t)$, then we have the average function value of the content in front of an arriving particle until time $t$.

From the fact that (\ref{eq:cumulative}) vanishes almost surely, it follows that if $W(s)$ has an ergodic distribution of some random variable $W$, that is 
\begin{equation*}
\frac{1}{t}\int_0^t g(W(s))ds\to Eg(W)=Ef(W+Y_e)\ ,
\end{equation*} 
then the long run average distribution of the content in front of particles will be distributed like $W+Y_e$ where $W,Y_e$ are independent. This is some generalised form of the well known PASTA (Poisson Arrivals See Time Averages) property for Poisson processes. See also Remark~2 of \cite{KB2013}.

We note that with
\begin{equation*}
\eta_0(\alpha)\equiv -\log\E e^{-\alpha J_0(1)}=c_0 \alpha+\int_{(0,\infty)}(1-e^{-\alpha x})\nu_0(dx)
\end{equation*}
we have, as in (4.8) of \cite{K1998}, that
\begin{equation}\label{eq:Ye_LJ_0T}
Ee^{-\alpha Y_e}=\frac{\eta_0(\alpha)}{\rho_0\alpha}\ .
\end{equation}

We conclude that if the process we are dealing with is regenerative with finite mean, nonarithmetic regeneration epochs, then the ergodic, stationary and limiting distributions all coincide. This will be the situation in the sequel.

All one needs to remember from this section is that under the assumptions that will soon appear, the limiting=stationary=ergodic distribution of the content in front of an arriving particle is that of $W+Y_e$ where $W,Y_e$ are independent, $W$ has the steady state distribution of the total content in the system while $Y_e$ has the excess distribution associated with the particular stream of particles (lowest priority) we are interested in.

\section{L\'{e}vy driven AP queue}\label{sec:model}
A single server processes the workload of $N$ types of particles at a constant rate of $r$ per unit of time. The AP rates of the particle classes are ordered: $b_1<b_2<\cdots<b_N$. Thus, the AP at time $t$ of a type $i$ particle that arrived at time $s$, $s \leq t$, is $b_i(t-s)$. 
The priority dynamics are illustrated in Figure \ref{fig:AP_example} for a two-class example.

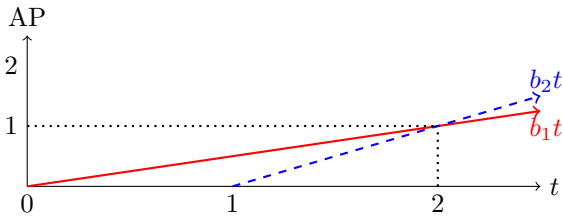
\begin{figure}[H]
\centering
\begin{tikzpicture}[xscale=2.7,yscale=0.8]
  \def\xmin{0}
  \def\xmax{2.5}
  \def\ymin{0}
  \def\ymax{2.5}
    \draw[->] (\xmin,\ymin) -- (\xmax,\ymin) node[right] {$t$} ;
    \draw[->] (\xmin,\ymin) -- (\xmin,\ymax) node[above] {AP} ;
    \foreach \x in {0,1,2}
    \node at (\x,\ymin) [below] {\x};
    \foreach \y in {1,2}
    \node at (\xmin,\y) [left] {\y};

    \draw[->,red, thick,domain=0:2.5]  plot (\x, {0.5*\x});
    \draw[->,blue,dashed, thick,domain=1:2.5]  plot (\x, {\x-1});
    \draw[smooth,black, thick,dotted]  (2,0) -- (2,1);
    \draw[smooth,black, thick,dotted]  (0,1) -- (2,1);

	\draw[] (2.4,0.95) node[right,red] {$b_1 t$};
	\draw[] (2.4,1.75) node[right,blue] {$b_2 t$};
	
\end{tikzpicture}
\caption{Example of the priority evolution for a type $1$ particle arriving at $t=0$ (solid red), and a type $2$ particle arriving at $t=1$ (dashed blue), with coefficients $b_1=0.5$ and $b_2=1$, respectively. At any time $t<2$ the type $1$ particle has higher priority than the type $2$ particle, but from $t=2$ and on he has lower priority and will effectively be overtaken.}\label{fig:AP_example}
\end{figure}

Suppose that $\{J_i(t),t\geq 0\}$, $i=1,\ldots,N$, are independent subordinators: almost surely monotone non-decreasing L\'{e}vy process. Denote by $\{\mathcal{F}_t\}_{t\geq 0}$ the natural filtration induced by $\{J_i:i=1,\ldots,N\}$. The total arrival process is then a subordinator as well,
\begin{equation*}
J(t)=\sum_{i=1}^N J_i(t).
\end{equation*}
Let $\rho_i=\E J_i(1)$, for all $i=1\ldots,N$ and assume, without loss of generality, that the workload is removed from the system (when it is not empty) at a linear rate of $r=1$. If $\rho=\sum_{i=1}^N\rho_i<1$ then the total workload is the reflection of a spectrally positive L\'{e}vy process with negative mean. Thus, the stationary distribution of the total workload satisfies the generalized PK formula (e.g. \cite{K2012}):
\begin{equation}\label{eq:GPK}
\E e^{-\alpha W_0}=\frac{\varphi'(0)\alpha}{\varphi(\alpha)}\ ,
\end{equation}
where
\begin{equation}\label{eq:J_total_LST}
\varphi(\alpha)=\alpha-\sum_{i=1}^N\eta_i(\alpha)\ ,
\end{equation}
and, in general,
\begin{equation*}
\eta_i(\alpha)=-\log Ee^{-\alpha J_i(1)}=c_i\alpha+\int_{(0,\infty)}(1-e^{-\alpha x})\nu_i(dx)\ .
\end{equation*}
$\nu_i$ is the L\'{e}vy measure satisfying $\int_{(0,\infty)}(x\wedge 1)\nu_i(dx)<\infty$. In fact, by the assumption that $\rho<1$ we actually have that $\int_{(0,\infty)}x\nu_i(dx)=\rho_i-c_i<\infty$.

\subsection{Lowest priority waiting time}\label{sec:overtaking}
For now we assume that $N=2$ and in the sequel we will show that the analysis for the lowest priority class can be extended in a straightforward manner to allow for any number of overtaking classes. Upon arrival to the queue, a type 2, i.e.\ low-priority, particle will have to wait for all work present in front of it to be completed, as well as for some work that will arrive after it.

From the point of view of type 2 particles, suppose that at the time of arrival such a particle finds a workload level of $W_0+Y_e=v$ in front of it from all priorities as well as those in front of it in a possible batch. All of this work will have to be served before this particle is processed. Recall that the cumulative priority 1 workload arriving during the first $t$ time units after our tagged particle arrives is $J_1(t)$, a subordinator such that $\E J_1(t)=\rho_1 t$, where $\rho_1<1$. At time $v$, the priority of the tagged particle will be $b_2v$, while for a type 1 particle arriving at time $0\le t\le v$ it will be $b_1(v-t)$. Thus for every $t$ such that $b_2v>b_1(v-t)$, the arriving type 1 workload will not overtake our tagged particle while for every $t$ such that $b_2v\le b_1(v-t)$ it will. This means that if we denote $a=1-b_2/b_1$, then the workload that arrives by time $av$ will surely preempt our tagged particle. Hence the waiting time of our tagged particle will surely be $v+J_1(av)$ instead of just $v$. Denote the decelerated input process by $J_a(t)=J_1(at)$.

Let us repeat the above argument once more to make what follows clearer. Since we know for certain that the waiting time is at least $v+J_a(v)$, then for the same reason, all workload arriving between $av$ and $a(v+J_a(v))$ will also preempt our tagged particle and thus the waiting time of the first particle will be at least
\begin{equation*}
v+J_a(v)+J_a(v+J_a(v))-J_a(v)=v+J_a(v+J_a(v))\ .
\end{equation*}

Repeating this argument, if we denote $T_0=v$ then we can define an iterative total overtaking workload process as a series of random variables,
\begin{equation}\label{eq:Tn_J}
T_{n+1}=v+J_a(T_n),\quad n\geq 1\ .
\end{equation}

\begin{lemma}\label{lemma:T_n}
$\{T_n:n\geq 0\}$ is a non-decreasing sequence of stopping times with respect to $\{\mathcal{F}_t\}_{t\geq 0}$.
\end{lemma}
\begin{proof}
Since $T_1=v+J_a(v)\ge v=T_0$, then by induction it follows that
\begin{equation*}
T_{n+1}=v+J_a(T_n)\ge v+J_a(T_{n-1})=T_n, \quad \forall n\geq 1\ .
\end{equation*}
Clearly, $T_0=v$ is a stopping time. If $T_n$ is a stopping time then recalling that $A\in\mathcal{F}_{T_n}$ if and only if $A\in\mathcal{F}$, $A\cap\{T_n\leq t\}\in\mathcal{F}_t$ $\forall t\geq 0$, we have that
\begin{equation*}
T_{n+1}=v+J_a(T_{n})\in\mathcal{F}_{T_{n}}\ ,
\end{equation*}
hence, since $T_n\leq T_{n+1}$ it follows that
\begin{equation*}
\{T_{n+1}\leq t\}=\{T_{n+1}\leq t\}\cap\{T_{n}\leq t\}\in\mathcal{F}_t, \quad t\geq 0\ .
\end{equation*}
\end{proof}

If we take $T=\lim_{n\to\infty}T_n$ then, since by Lemma \ref{lemma:T_n},
\begin{equation*}
\E T_n=v+\E J_a(T_{n-1})=v+\rho_1a\E T_{n-1}\ ,
\end{equation*}
then
\begin{equation*}
\E T_n=(\rho_1 a)^n+\frac{1-(\rho_1a)^n}{1-\rho_1a}v\ ,
\end{equation*}
and thus
\begin{equation}\label{eq:mean_T}
\E T=\frac{v}{1-\rho_1a}\ .
\end{equation}

Observe that if $J_1(t)$ is compound Poisson then taking expectation of \eqref{eq:mean_T} with respect to the workload upon arrival $W_0$, yields Kleinrock's expected waiting time formula for the lowest priority class (see (3.47) in \cite{book_K1976}),
\begin{equation*}
\E W_2=\E[\E(T|W_0)]=\frac{\frac{\rho\E X_e}{1-\rho}}{1-\rho_1\left(1-\frac{b_2}{b_1}\right)}\ ,
\end{equation*}
where $X_e$ is the residual service time.

From \eqref{eq:mean_T} we conclude that $T<\infty$ a.s. for any initial $v\geq 0$. Furthermore, by Proposition~7 (p. 21) of \cite{book_B1996} we have that $J_a(T_n)\to J_a(T)$ a.s. and hence $v+J_a(T)=T$ a.s. Moreover, if $T_{n-1}<T_n$ then for $T_{n-1}\le t<T_n$ we have that
\begin{equation*}
v+J_a(t)\ge v+J_a(T_{n-1})=T_n>t\ ,
\end{equation*}
so that in particular $v+J_a(t)>t$ for all $0\le t<T$. Therefore, $T$, which is itself a stopping time, is a.s. equal to
\begin{equation*}
\inf\{t\ge 0|v+J_a(t)-t=0\}\ .
\end{equation*}

Thus, as is well known (e.g. p. 81 of \cite{book_K2006}),
\begin{equation*}
\E e^{-\alpha T}=e^{-\varphi_a^{-1}(\alpha)v}\ ,
\end{equation*}
where $\varphi_a(\alpha)=\log \E e^{-\alpha(J_a(1)-1)}=\alpha-a\eta_1(\alpha)$.

All type 1 workload that arrives by time $aT$ will be served before our tagged particle and all type 1 workload that arrives after $aT$ will be served after. In particular, if instead of $v$ we insert the total workload of $W_0+Y_e$, then applying the generalized P-K formula \eqref{eq:GPK} and the LST of $Y_e$ with respect to $J_2$, as defined in \eqref{eq:Ye_LJ_0T}, at $\varphi_a^{-1}(\alpha)$ yields the LST of the unconditional steady state waiting time of type $2$ particles. This brings us to the main result of this article.

\begin{theorem}\label{thm:2class}
The LST of the steady state waiting time of class $2$ particles is given by
\begin{equation*}\label{eq:W2_LST}
\E e^{-\alpha W_2}=\frac{\varphi'(0)\varphi_a^{-1}(\alpha)}{\varphi\left(\varphi_a^{-1}(\alpha)\right)}\cdot \frac{\eta_2(\varphi_a^{-1}(\alpha))}{\rho_2 \varphi_a^{-1}(\alpha)}=
\frac{\varphi'(0)\eta_2(\varphi_a^{-1}(\alpha))}{\rho_2\varphi\left(\varphi_a^{-1}(\alpha)\right)}\ ,
\end{equation*}
where
\begin{equation*}
\varphi_a(\alpha)=\log \E e^{-\alpha(J_a(1)-1)}\ ,
\end{equation*}
and
\begin{equation*}
\eta_2(\alpha)=-\log \E e^{-\alpha J_2(1)}\ .
\end{equation*}
\end{theorem}

For the M/G/1 queue with arrival rate $\lambda=\lambda_1+\lambda_2$ and traffic intensity $\rho=\lambda_1 \E X_1+\lambda_2 \E X_2$, the LST of the steady state customer (not particle) waiting time is given by the P-K formula
\begin{equation*}
\frac{1-\rho}{1-\rho\E e^{-\alpha X_e}}\ ,
\end{equation*}
hence in this case the waiting time of the first type-2 particle in a batch, which is the waiting time of an arriving type-2 customer, is given by
\begin{equation}\label{eq:LST_MG1_W2}
\E e^{-\alpha W_2}=\frac{1-\rho}{1-\rho \E e^{-\varphi_a^{-1}(\alpha)X_e}}\ ,
\end{equation}
where
\begin{align*}
\varphi_a(\alpha) &= \log \E e^{-\alpha(J_a(1)-1)}=\alpha-\lambda_1a(1-\E e^{-\alpha X})\\
&= \alpha(1-\rho_1a\E e^{-\alpha X_e})\ .
\end{align*}
By applying some algebra it can be verified that \eqref{eq:LST_MG1_W2} coincides with the lowest priority class LST established in (31) of \cite{STZ2014} for the compound Poisson input case, which used a different construction.

If for some $n$ we have that $T_n=T_{n-1}$ then also $T_m=T_{n-1}$ for all $m\ge n-1$, so that $T=T_{n-1}=T_n$. We next assert that the number of such overtaking increments is finite almost surely if and only if the input process of type $1$ particles is Compound Poisson. The proof is given in the appendix.
\begin{proposition}\label{prop:finite_n}
Let $K=\inf\{n\geq 0:T_n=T_{n-1}\}$. For any $v>0$, if $\{J_1(t):t\geq 0\}$ is a compound Poisson process then $\P(K<\infty)=1$, otherwise $\P(K=\infty)=1$.
\end{proposition}

\subsection{Joint distributions}\label{sec:joint}
In a two-station tandem fluid queue with L\'evy input the joint distribution of the queue sizes $(Q_1,Q_2)$ can be derived, in part by using the fact that
\begin{equation*}
(Q_1,Q_2)\sim(Q-Q_2,Q_2)\ .
\end{equation*}
where $Q=Q_1+Q_2$ is the total system size (e.g. p. 174 of \cite{book_DM2015}). In our model however we do not have a clear relation between the waiting times $W_1$, $W_2$ and the total workload $W_0$, except that the expected values satisfy a work conserving equation (see p. 114 of \cite{book_K1976}).

A joint distribution that is straightforward to obtain, because we know the conditional distribution, is that of the initial workload faced by a low priority particle and his total waiting time, i.e.\ $(W_2,W_0+Y_e)$,
\begin{align*}
\E e^{-\alpha W_2-\beta(W_0+Y_e)} &= \E e^{-(\varphi_a^{-1}(\alpha)+\beta)(W_0+Y_e)}\\
&= \frac{\varphi'(0)\eta_2(\varphi_a^{-1}(\alpha)+\beta)}{\rho_2\varphi\left(\varphi_a^{-1}(\alpha)+\beta\right)}\ . 
\end{align*}

\subsection{Lowest priority among multiple classes}\label{sec:multi}
Let us now consider the general model with $N$ particle classes with independent subordinator input processes $\{J_i(t):i=1,\ldots,N\}$, with $\rho_i=\E J_i(1)$. Recall that the total arrival process is then a subordinator as well, denoted by $J(t)=\sum_{i=1}^N J_i(t)$. If $\rho=\sum_{i=1}^N\rho_i<1$ then the total workload is the reflection of a spectrally positive L\'{e}vy process with negative mean. Thus, the distribution of workload upon arrival of any particle satisfies the generalized PK formula \eqref{eq:GPK} with $\varphi(\alpha)$ as defined in \eqref{eq:J_total_LST}.

Let $a_i=1-\frac{b_N}{b_i}$ for $i=1,\ldots,N-1$, where $b_1<b_2<\cdots<b_N$ are the AP rates of all classes. If a type $N$ particle arrives when there is a total workload of $v$ in the system then all workload of type $i<N$ that arrives during $[0,a_i v)$ will surely overtake him. Let
\begin{equation*}
J_a(t)=J_1(a_1 t)+J_2(a_2 t)+\ldots+J_{N-1}(a_{N-1} t)\ ,
\end{equation*}
be the total workload overtaking the tagged particle during an interval of length $t$. Repeating the arguments of the previous section we have that \eqref{eq:Tn_J} still holds: $T_n=v+J_a(T_{n-1})$ $\forall n\geq 1$ is a series of non-decreasing stopping times, where $T_0=v$, and
\begin{align*}
\E T_n &= v+\E J_a(T_{n-1})=v+\sum_{i=1}^{N-1}\E J_i(a_i T_{n-1})\\
&=v+\sum_{i=1}^{N-1}\rho_i a_i \E T_{n-1} \\
&=v\left(1+\sum_{i=1}^{N-1}\rho_i a_i\right)+\sum_{i=1}^{N-1}\rho_i a_i\sum_{j=1}^{N-1}\rho_j a_j \E T_{n-2} \\
&=v\left(1+\sum_{i=1}^{N-1}\rho_i a_i+\left[\sum_{i=1}^{N-1}\rho_i a_i\right]^{2}\right)\\
&+\left[\sum_{i=1}^{N-1}\rho_i a_i\right]^{2}\sum_{j=1}^{N-1}\rho_j a_j \E T_{n-3}  \\
\vdots & \\
&= v\frac{1-\left[\sum_{i=1}^{N-1}\rho_i a_i\right]^{n-1}}{1-\sum_{i=1}^{N-1}\rho_i a_i}+\left[\sum_{i=1}^{N-1}\rho_i a_i\right]^{n-1}\ .
\end{align*}
If $T=\lim_{n\to\infty}T_n$ then
\begin{equation*}
\E T=\frac{v}{1-\sum_{i=1}^{N-1}\rho_i a_i}<\infty \ ,
\end{equation*}
again coinciding with (3.47) in \cite{book_K1976} when taking expectation with respect to $W_0$.

As $J_i(a_i t)$ are all L\'{e}vy processes, with respect to the time change $a_i t$, then so is their superposition $J_a(t)$. We can therefore define
\begin{equation*}
\varphi_a(\alpha)=\log\E e^{-\alpha(J_a(1)-1)}=\alpha-\sum_{i=1}^{N-1}a_i\eta_i(\alpha)\ .
\end{equation*}
Repeating the arguments of the previous section we have that the waiting time of a type $N$ particle is distributed as $T=\inf\{t\geq 0:v+J_a(t)-t=0\}$, and we obtain $\E e^{-\alpha T}=e^{-\varphi_a^{-1}(\alpha)v}$. Finally, \eqref{eq:Ye_LJ_0T} for the input of type $N$ (with exponent function $\eta_N$) and \eqref{eq:GPK} for the total input (with exponent function $\varphi$) yield our general result.

\begin{theorem}\label{thm:LST_General}
The LST of the steady state waiting time of class $N$ particles is
\begin{equation*}\label{eq:LST_WN}
\E e^{-\alpha W_N}=
\frac{\varphi'(0)\eta_N(\varphi_a^{-1}(\alpha))}{\rho_N\varphi\left(\varphi_a^{-1}(\alpha)\right)}\ .
\end{equation*}
\end{theorem}

\section{Concluding remarks}\label{sec:conclusion}
We have presented here a new method for deriving the previously unknown waiting time distribution of the lowest-priority particle class in an accumulating priority L\'{e}vy driven queue, which generalizes the known result for the M/G/1 queue. Constructing a similar method for the highest priority class, and ultimately for any priority class, is an open challenge. This will require the construction of a negative workload overtaking process that is removed from the initial amount of work upon arrival. To this end perhaps approximations are in order. 

The method presented here also applies to non-linear accumulation functions that satisfy certain conditions, similar to those that appeared in \cite{L2015} for the M/G/1 analysis. For a wider class of continuous functions the method can be applied partially: the first-passage time problem can be formulated for the overtaking process, but the time change of the input process is no longer linear and it is unclear if it is at all possible to analytically derive the LST of the waiting time.

\section*{Appendix}
\begin{proof}[Proof of Proposition \ref{prop:finite_n}]
If $J_1(t)$ is not a compound Poisson process then it has an infinite number of infinitesimal jumps during every finite interval. Therefore the probability of zero jumps during an increment period, i.e.
\begin{equation*}
\P(T_{n+1}-T_n=0)=\P(J_a(T_n)-J_a(T_{n-1})=0)\ ,
\end{equation*}
is zero for any $n\geq 1$. And thus $K=\infty$ a.s.

If $J_1(t)$ is a compound Poisson process with $\rho_1<1$ then we will show that
\begin{equation*}
\P \big(\{ T_n=T_{n-1}, \ \mathrm{i.o.}\}\big)=1\ .
\end{equation*}
This can be argued by showing that
\begin{equation*}
\sum_{n=1}^\infty \P(T_n-T_{n-1}>0)<\infty\ ,
\end{equation*}
and applying the Borel-Cantelli Lemma.

The probability of no overtaking in the first interval is given by
\begin{equation*}
\P(T_1-T_0=0)=\P(J_a(v)=0)=e^{-\lambda_1 av}\ .
\end{equation*}
For any $n\geq 1$ the additional workload at step $n$ is
\begin{equation*}
T_{n+1}-T_{n}=J_a(T_{n})-J_a(T_{n-1})\ ,
\end{equation*}
which requires more careful treatment as the increments are not independent of the input process. Nevertheless, we argue that
\begin{equation*}
J_a(T_{n})-J_a(T_{n-1})\sim \tilde{J}_a(T_{n-1}-T_{n-2}),\quad n\geq 1\ .
\end{equation*}
where $\tilde{J}_a$ is an independent copy of $J_a$. By the strong Markov property it is well known that for any a.s. finite stopping time $S$, $\mathcal{F}_S$ is independent of $\{Y(S+t)-Y(S):t\geq 0\}$, where $Y$ is a L\'evy process. Thus, if $U\in\mathcal{F}_S$ then $U-S\in\mathcal{F}_S$ as well, and $U-S$ is independent of the increments of $Y$ after $S$.

Indeed, $J_a$ is a process with independent increments, and by Lemma \ref{lemma:T_n} $T_n$ and $T_{n-1}$ are stopping times such that $T_n\in\mathcal{F}_{T_{n-1}}$. We conclude that $T_n-T_{n-1}$ is independent of $\{J_a(T_{n-1}+t)-J_a(T_{n-1}):t\geq 0\}$ for any $n\geq 1$, and therefore
\begin{align*}
\E e^{-\alpha(T_{n+1}-T_{n})} &= \E e^{-\alpha(J_a(T_{n-1}+T_n-T_{n-1})-J_a(T_{n-1}))}\\
&= \E e^{-\eta_a(\alpha)(J_a(T_{n}-T_{n-1}))}\ ,
\end{align*}
where $\eta_a(\alpha)=-\log\E e^{-\alpha J_a(1)}$. For example,
\begin{align*}
\E e^{-\alpha\big(T_2-T_1\big)} &= \E e^{-\alpha(J_a(T_{1}-T_{0}))}= \E e^{-\alpha(J_a(J_a(v)))}\\
&= \E e^{-\eta_a(\alpha)J_a(v)}=e^{-\eta_a(\eta_a(\alpha))v}\ .
\end{align*}
That is, if we condition on $J_a(v)$ then the additional workload in the second period is distributed as the Compund Poisson process $J_a$ on an interval of length $J_a(v)$. Using the convexity of $\eta_a$, this yields
\begin{align*}
\P(T_2-T_1=0) &= \E\big[\P(T_2-T_1=0|J_a(v))\big]\\
&= \E e^{-\lambda_1 aJ_a(v)}=e^{-\eta_a(\lambda_1 a)v}\geq e^{-\lambda_1 a^2\rho_1 v}\ .
\end{align*}
By iterating we have that
\begin{equation*}
\P(T_n-T_{n-1}=0)=e^{-\eta_a(\eta_a(\ldots\eta_a(\lambda_1 a)\ldots))}\geq e^{-\lambda_1 a(a\rho_1)^{n-1}v}\ .
\end{equation*}

Finally, this leads us to:
\begin{equation*}
\begin{split}
\sum_{n=1}^\infty \P(T_n-T_{n-1}>0)&=\sum_{n=1}^\infty \big(1-\P(T_n-T_{n-1}=0)\big) \\
&\leq \sum_{n=1}^\infty \big(1-e^{-\lambda_1 a(a\rho_1)^{n-1}v}\big) \\
&\leq \sum_{n=1}^\infty\lambda_1 a(a\rho_1)^{n-1}v=\frac{\lambda_1 a v}{1-a\rho_1}<\infty\ ,
\end{split}
\end{equation*}
from which it follows that $\P(\{T_n>T_{n-1}, \ \mathrm{i.o.}\})=0$.
\end{proof}

\section*{Acknowledgements}
We wish to thank an anonymous referee whose helpful comments helped improve this paper. This work was supported in part by grant 1462/13 from the Israel Science Foundation and the Vigevani Chair in Statistics.

\section*{References}
\bibliography{BigBib}

\begin{thebibliography}{10}

\bibitem{book_B1996}
J.~Bertoin.
\newblock {\em L{\'e}vy processes}, volume 121.
\newblock Cambridge university press, 1996.

\bibitem{DM2012}
K.~D\k{e}bicki and M.~Mandjes.
\newblock L{\'e}vy-driven queues.
\newblock {\em Surveys in Operations Research and Management Science}, 17(1):15
  -- 37, 2012.

\bibitem{book_DM2015}
K.~D\k{e}bicki and M.~Mandjes.
\newblock {\em Queues and L{\'e}vy fluctuation theory}.
\newblock Springer, 2015.

\bibitem{FD2015}
V.~A. Fajardo and S.~Drekic.
\newblock Waiting time distributions in the preemptive accumulating priority
  queue.
\newblock {\em Methodology and Computing in Applied Probability}, pages 1--30,
  2015.

\bibitem{G1977}
H.~M. Goldberg.
\newblock Analysis of the earliest due date scheduling rule in queueing
  systems.
\newblock {\em Mathematics of Operations Research}, 2(2):145--154, 1977.

\bibitem{HR2016}
M.~Haviv and L.~Ravner.
\newblock Strategic bidding in an accumulating priority queue: equilibrium
  analysis.
\newblock {\em Annals of Operations Research}, 244(2):505--523, 2016.

\bibitem{K1998}
O.~Kella.
\newblock {An exhaustive L{\'e}vy storage process with intermittent output}.
\newblock {\em Stochastic Models}, 14(4):979--992, 1998.

\bibitem{K2012}
O.~Kella.
\newblock {The class of distributions associated with the generalized
  Pollaczek-Khinchine formula}.
\newblock {\em Journal of Applied Probability}, 49(3):883--887, 09 2012.

\bibitem{KB2013}
O.~Kella and O.~Boxma.
\newblock {Useful martingales for stochastic storage processes with
  L{\'e}vy-type input}.
\newblock {\em Journal of Applied Probability}, 50(2):439--449, 2013.

\bibitem{KY2013}
O.~Kella and M.~Yor.
\newblock {Unifying the Dynkin and Lebesgue-Stieltjes formulae}.
\newblock {\em Journal of Applied Probability}, 54, 2017 In Press.

\bibitem{K1964}
L.~Kleinrock.
\newblock A delay dependent queue discipline.
\newblock {\em Naval Research Logistics Quarterly}, 11(3-4):329--341, 1964.

\bibitem{book_K1976}
L.~Kleinrock.
\newblock {\em {Queueing Systems}}, volume II: Computer Applications.
\newblock Wiley, 1976.

\bibitem{KF1967}
L.~Kleinrock and R.~P. Finkelstein.
\newblock Time dependent priority queues.
\newblock {\em Operations Research}, 15(1):104--116, 1967.

\bibitem{book_K2006}
A.~Kyprianou.
\newblock {\em Introductory lectures on fluctuations of L{\'e}vy processes with
  applications}.
\newblock Springer, 2006.

\bibitem{L2015}
N.~Li.
\newblock Recent advances in accumulating priority queues.
\newblock {\em Electronic Thesis and Dissertation Repository, Paper 3401},
  2015.

\bibitem{LS2016}
N.~Li and D.~A. Stanford.
\newblock Multi-server accumulating priority queues with heterogeneous servers.
\newblock {\em European Journal of Operational Research}, 252(3):866 -- 878,
  2016.

\bibitem{NA1979}
A.~Netterman and I.~Adiri.
\newblock A dynamic priority queue with general concave priority functions.
\newblock {\em Operations Research}, 27(6):1088--1100, 1979.

\bibitem{SB2013}
V.~Sarhangian and B.~Balcıo{\u{g}}lu.
\newblock A first passage time problem for spectrally positive {L\'{e}}vy
  processes and its application to a dynamic priority queue.
\newblock {\em Operations Research Letters}, 41(6):659 -- 663, 2013.

\bibitem{SSTZ2014}
A.~B. Sharif, D.~A. Stanford, P.~Taylor, and I.~Ziedins.
\newblock A multi-class multi-server accumulating priority queue with
  application to health care.
\newblock {\em Operations Research for Health Care}, 3(2):73--79, 2014.

\bibitem{STZ2014}
D.~A. Stanford, P.~Taylor, and I.~Ziedins.
\newblock Waiting time distributions in the accumulating priority queue.
\newblock {\em Queueing Systems}, 77(3):297--330, 2014.

\end{thebibliography}

\end{document}